\documentclass[12pt,reqno,oneside]{amsart}
      
\usepackage{amsmath,amsthm,amsfonts,amssymb}
\usepackage{indentfirst}
\usepackage{url}
\usepackage{hyperref}
\usepackage{graphicx}
\usepackage[usenames]{color}

\usepackage{pgfplots} 

\theoremstyle{plain}
\newtheorem{theorem}{Theorem}[section]
\newtheorem{cor}[theorem]{Corollary}

\theoremstyle{definition}

\newtheorem{exa}[theorem]{Example}
\newtheorem{obs}[theorem]{Remark}

\numberwithin{equation}{section}

\newcommand{\Cind}{C_{\mathrm{ind}}}

\begin{document}

\baselineskip=18pt

\title[]{Immigration Processes with Binomial Catastrophes and Random Survival Parameters}

\author[Sandro Gallo]{Sandro Gallo}
\address[Sandro Gallo]{Departamento de Estat\'istica, Universidade Federal de S\~ao Carlos,  Brazil.}
\email{sandro.gallo@ufscar.br}

\author[Alejandro Rold\'an-Correa]{Alejandro~Rold\'an-Correa}
\address[Alejandro Rold\'an]{Instituto de Matem\'aticas, Universidad de Antioquia, Colombia}
\email{alejandro.roldan@udea.edu.co}

\thanks{Sandro Gallo was supported by Fapesp (23/13453-5)  and Alejandro Roldan by Universidad de Antioquia (Project No. 2025-80410).}

\date{\today}

\begin{abstract}
We consider immigration processes with binomial catastrophes and random survival parameters. Two sources of randomness are analyzed. In the first model, the survival parameter is independently resampled at each catastrophe. In the second model, individuals are assigned independent survival parameters at birth, which remain fixed over time. We show that the first model exhibits almost sure extinction, as in the classical case with a fixed survival parameter. In contrast, the second model exhibits a phase transition, admitting survival with positive probability, depending on the distribution of individual survival parameters. We provide explicit formulas for both the survival probability and the expected time to extinction. Finally, our proofs establish a novel methodological bridge with the Firework process, unifying population dynamics with spatial models of information spreading.
\end{abstract}

\maketitle

\section{Introduction}
\label{S: Introduction}

A widely used model in population dynamics is the immigration process with binomial catastrophes (hereafter abbreviated as IPBC), in which each individual present at the time of a catastrophe survives independently with a given probability. Such models have been investigated in various probabilistic settings and exhibit robust extinction behavior; see \cite{AEL2007, economou2004, JMR2016} and the references therein. However, populations exposed to disasters, such as those arising in metapopulation models or epidemic dynamics, often display long-term persistence. This raises the question of what mechanisms can counterbalance catastrophic effects; see \cite{DJMR2023, JMR2016, MRS2015}.

For the IPBC with a fixed survival parameter strictly smaller than one, extinction occurs with probability one, independently of the immigration rate. This result, established in several works including Artalejo et al. \cite{AEL2007}, shows that the cumulative effect of recurrent and equally severe catastrophes cannot be balanced by immigration in a population with deterministic fitness. 

In this paper, we investigate two natural extensions of the classical model obtained by introducing randomness into the survival mechanism. The two models differ in the level at which randomness is incorporated. In the first model, the survival parameter associated with each catastrophe is independently sampled from a probability distribution supported on the unit interval. This formulation is still in the realm of binomial catastrophes, randomly choosing the severity of catastrophic events over time. In the second model, randomness is assigned at the individual level: each individual is endowed, at the time of its creation, with a {survival parameter} drawn independently from a given distribution, which remains fixed throughout all subsequent catastrophes. This is no longer a binomial catastrophe model; however, each individual is still independently subject to catastrophic events that affect them differently, depending on their respective fitness.

Interestingly, these two sources of randomness lead to qualitatively different long-term behaviors. For the catastrophe-random model, we show that extinction still occurs with probability one, extending the classical result for fixed survival parameters. In contrast, the individual-random model exhibits a phase transition, admitting survival with positive probability in certain cases. For both models, our main results provide explicit formulas for the expected time to extinction, alongside the exact survival probability for the individual-random case. For the individual-random model, our criteria reveal an interesting fact: while ultimate survival depends on both the immigration rate and the distribution of survival parameters, the condition under which the expected number of catastrophes is infinite does not depend on the immigration rate.

A possible interpretation of the above results, from the ecological perspective, would be that populations with diverse and permanent survival traits are more resilient to recurrent catastrophes than populations subject to population-wide fluctuations in disaster severity.

Beyond the above results, this work also contributes to establishing a formal bridge between general immigration processes with catastrophes and spatial models of information spreading. Specifically, we introduce a spatial interpretation of the IPBC, viewing it as a model for information propagation on $\mathbb{N}$, known in the literature as the \emph{Firework process} \cite{GGJR2014, JMZ2011} (shortened to FP in the sequel). This connection unifies two previously distinct bodies of literature, allowing results from spatial rumor processes to be applied to population dynamics problems. As a byproduct, we derive an explicit formula for the expected final range of the rumor, which translates directly into the expected time to extinction for the population model.

A few words regarding our proofs. For the catastrophe-random model, we rely on the classical Foster criterion to establish ergodicity, while the explicit formula for the expected time to extinction is derived via the Poisson transform and boundary analysis of the Kolmogorov equations. For the {individual-random} model, the proof of our main theorem exploits the aforementioned bridge with the Firework process.

The paper is organized as follows. Section \ref{sec:results} contains the definition of the IPBC and the main results about it. We define the FP in Section \ref{sec:bridge} and explain how it relates to the IPBC, and state some consequences of this relation. Finally, we prove our results in Section \ref{sec:proofs}.

\section{Models and Main Results}\label{sec:results}

\subsection{Immigration process with binomial catastrophes (IPBC)} 
We consider a population of individuals living in a single colony and subject to binomial catastrophes. The colony produces new individuals according to a Poisson process with rate $\lambda > 0$, while catastrophes occur independently according to a Poisson process with rate 1. Specifically, the population size at time $t$ is modeled by a continuous-time Markov process $\{X(t): t\ge0\}$ with initial condition $X(0)=1$. We denote this process by $C(\lambda,p)$. At each catastrophe, individuals survive independently with a fixed probability (survival parameter) $p\in(0,1)$.

The infinitesimal generator of $X(t)$ is given by
\[
q_{ij} =
\begin{cases}
\lambda, & j = i + 1, \\
\binom{i}{j} p^{j}(1-p)^{i-j}, & j = 0,\dots,i-1, \\
-(\lambda + 1), & j = i, \\
0, & \text{otherwise}.
\end{cases}
\]

We say that the process $C(\lambda,p)$ \emph{survives} if
\begin{equation}\label{def_sobrevivencia}
\mathbb P(\mathcal A_C) :=\mathbb P\big(X(t)>0 \text{ for all } t>0\big)>0.    
\end{equation}
Otherwise, the process is said to \emph{die out}. In either case, let $M_C$ and $T_C$ denote the total number of catastrophes and the time elapsed until extinction, respectively.

\begin{obs}\label{E(T)=E(M)} We can write $T_C=\sum_{i=1}^{M_C} T_i$, where the random variables $T_i$ are independent and identically distributed (i.i.d.) with an $\text{Exp}(1)$ distribution. Here, $T_1$ represents the time of occurrence of the first catastrophe, and $T_i$ is the inter-arrival time between the $(i-1)^{\text{th}}$ and $i^{\text{th}}$ catastrophes for $i=2,\ldots, M_C$. Since the variables $T_i$ are i.i.d. with $\mathbb{E}(T_i)=1$, and the event $\{M_C \ge i\}$ depends only on the history of the process prior to the $i^{\text{th}}$ catastrophe (making it independent of $T_i$), it follows from Wald's identity that $\mathbb{E}(T_C) = \mathbb{E}(M_C).$ 
\end{obs}

It is a known result \cite{AEL2007} that the process $C(\lambda,p)$ dies out for all $\lambda>0$ and $p \in (0,1)$. Furthermore, the expected time to extinction is given by
\begin{equation}\label{Tiemp_Medio_Artalejo}
    \mathbb{E}(T_C) = \frac{1}{\lambda} \left( \prod_{k=0}^{\infty} \left(1 + \lambda p^k \right) - 1 \right) < \infty.
\end{equation}

In particular, this shows that the immigration mechanism is insufficient to counterbalance the effect of recurrent binomial catastrophes when the survival parameter is fixed.

\subsection{Main results: How to help the population survive?}
There is a shared interest in probability theory and population dynamics in the phenomenon of phase transition: models in which an abrupt change of macroscopic behavior is observed when a parameter changes at the microscopic level. In this respect, the IBPC is not very interesting — no matter how close to 1 the survival parameter $p$ is, the population dies out. A natural question is therefore: what does it take for the population to survive?
To answer this, we explore two independent approaches to introducing randomness into the survival mechanism.

\subsubsection{The catastrophe-random model}
Our first approach is to randomly draw the survival parameter at each catastrophe, with the hope that the random occurrence of less severe catastrophes will help the population persist. In this variant, the random survival parameter, denoted $X$, is therefore tied to the catastrophic events themselves: each time a catastrophe occurs, $X$ is sampled independently from a probability distribution $\nu$ supported on $[0,1]$, and, conditional on $X=p$, individuals survive independently according to a binomial thinning with parameter~$p$.

We denote this process by $C_{\mathrm{cat}}(\lambda,\nu)$ and refer to it as the \emph{catastrophe-random model}.

\begin{theorem}\label{Th:Catastrofe_aleatoria}
Consider the process $C_{\mathrm{cat}}(\lambda,\nu)$ with $\lambda>0$, and suppose that the probability distribution $\nu$ on $[0,1]$ is not concentrated at $1$. Then:
\begin{enumerate}
    \item The process dies out almost surely.
    \item The expected number of catastrophes until extinction (which equals the expected time to extinction, see Remark~\ref{E(T)=E(M)}) is finite and given by
    \begin{equation}\label{eq:exp_cat_random}
    \mathbb E(M_{C_{\mathrm{cat}}}) = -\frac{1}{\lambda} \frac{\mathcal{S}_\nu(\lambda)}{1 + \mathcal{S}_\nu(\lambda)} < \infty,
    \end{equation}
    where $\mathcal{S}_\nu(\lambda)$ is the analytic continuation of the alternating series
    \begin{equation}\label{Eq:aux}
    \mathcal{S}_\nu(\lambda) = \sum_{n=1}^\infty \frac{(-\lambda)^n}{\prod_{j=1}^n (1-\mathbb{E}[X^j])},
    \end{equation}
    with $X \sim \nu$.
\end{enumerate}
\end{theorem}

A detailed discussion regarding the exact domain of convergence of the series \eqref{Eq:aux} and its analytic continuation is deferred to Remark \ref{Rem:AnalyticContinuation} in Section \ref{sec:proofs}. 
If we evaluate \eqref{Eq:aux} for the classical case of a fixed survival parameter (i.e., $\nu = \delta_p$ for some $p \in (0,1)$), the term in the denominator becomes $\prod_{j=1}^n (1-p^j)$. By applying Euler's classical identity for $q$-hypergeometric series \cite[Eq. 1.3.15]{GR2004}, the algebraic fraction $\frac{-\mathcal{S}_{\delta_p}(\lambda)}{1+\mathcal{S}_{\delta_p}(\lambda)}$ analytically simplifies to
\[
\prod_{k=0}^\infty (1+\lambda p^k) - 1.
\]
Substituting this back into \eqref{eq:exp_cat_random} recovers exactly formula \eqref{Tiemp_Medio_Artalejo}.

\begin{exa}[Power-function distribution]\label{Ex:PowerFunction}
When the survival parameter follows a power-function distribution on $(0,1)$, display \eqref{eq:exp_cat_random} admits a simple form in terms of the parameters. A random variable $X$ has a power-function distribution if its probability density function is given by $f(x) = ax^{a-1}$ for $x\in(0,1)$, where $a>0$ is a parameter. In this case, the moments are $\mathbb{E}[X^j] = \frac{a}{a+j}$, which implies
$$1 - \mathbb{E}[X^j] = \frac{j}{a+j}.$$
The product in the denominator simplifies
$$\prod_{j=1}^n (1-\mathbb{E}[X^j]) = \prod_{j=1}^n \frac{j}{a+j} = \frac{n!}{(a+1)_n},$$
where $(a+1)_n = (a+1)(a+2)\dots(a+n)$ is the Pochhammer symbol. Substituting this into the alternating series yields
$$\mathcal{S}_\nu(\lambda) = \sum_{n=1}^\infty \frac{(a+1)_n}{n!} (-\lambda)^n.$$
This is precisely the binomial series expansion for $(1+\lambda)^{-(a+1)}$ minus the $n=0$ term. Although this series representation converges strictly for $\lambda < 1$, its analytic continuation extends to all $\lambda > 0$ (as formally justified in Remark \ref{Rem:AnalyticContinuation}), yielding the well-defined algebraic function
$$\mathcal{S}_\nu(\lambda) = (1+\lambda)^{-(a+1)} - 1.$$
Plugging this result back into \eqref{eq:exp_cat_random}, we find a striking closed-form expression for the expected number of catastrophes:
$$\mathbb E(M_{C_{\mathrm{cat}}}) = \frac{(1+\lambda)^{a+1} - 1}{\lambda}.$$
The behavior of this expected time as a function of the immigration rate $\lambda$ is illustrated in Figure \ref{fig:ExpectedCatastrophes} for various values of the parameter $a$.
\end{exa}

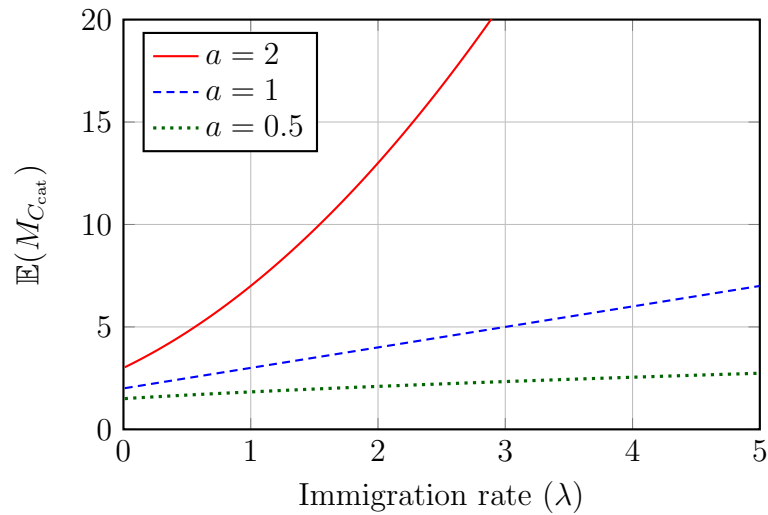
\begin{figure}[htpb]
    \centering
    \begin{tikzpicture}
        \begin{axis}[
            width=10cm,
            height=7cm,
            xlabel={Immigration rate ($\lambda$)},
            ylabel={$\mathbb{E}(M_{C_{\mathrm{cat}}})$},
            xmin=0, xmax=5,
            ymin=0, ymax=20,
            domain=0.01:5,
            samples=100,
            grid=major,
            legend pos=north west,
            legend cell align={left},
            thick
        ]
        \addplot[color=red] {((1+x)^3 - 1)/x};
        \addlegendentry{$a=2$ }
        \addplot[color=blue, densely dashed] {x + 2};
        \addlegendentry{$a=1$ }
        \addplot[color=black!60!green, dotted, very thick] {((1+x)^1.5 - 1)/x};
        \addlegendentry{$a=0.5$ }      
        \end{axis}
    \end{tikzpicture}
    \caption{Expected number of catastrophes until extinction for the $C_{\mathrm{cat}}$ model as a function of the immigration rate $\lambda$, assuming the survival parameter follows a power-function distribution with parameter $a$. The chart illustrates the linear growth $\lambda + 2$ for the uniform case ($a=1$), contrasting with the polynomial growth for $a=2$ and the sublinear behavior for $a=0.5$.}\label{fig:ExpectedCatastrophes}
\end{figure}

As shown by Theorem~\ref{Th:Catastrofe_aleatoria} and Example~\ref{Ex:PowerFunction}, introducing randomness at the macroscopic level of the catastrophic events only alters the expected time to extinction, but fails to prevent the ultimate collapse of the population. 

We now explore an alternative survival mechanism.

\subsubsection{The individual-random model}
Motivated by the recent work \cite{CM2025}, we define a second variant in which randomness is introduced at the microscopic level of the individuals. Upon creation, each individual is independently assigned a survival parameter denoted by $X$ drawn from a probability distribution $\nu$ on $[0,1]$. This survival probability remains fixed for that individual throughout all subsequent catastrophes. At each catastrophe, individuals survive independently according to their own assigned survival parameters.

We denote this process by $\Cind(\lambda,\nu)$ and refer to it as the \emph{individual-random model}.

\begin{theorem}\label{Th:Particula_aleatoria}
The survival probability of the process $\Cind(\lambda,\nu)$ is given by
\begin{equation}\label{eq:prob_surv}
\mathbb P(\mathcal A_{\Cind})
=
\frac{1+\lambda}
{\lambda+\lambda\sum_{j\ge1}
\prod_{k=0}^{j-1}
\frac{1}{1+\lambda\,\mathbb E(X^{k+1})}},
\end{equation}
where $X\sim\nu$. (The probability is $0$ if the series of the denominator diverges). Moreover, the expected number of catastrophes until extinction (which equals the expected time to extinction, see Remark~\ref{E(T)=E(M)}) is given by
\begin{equation}\label{eq:exp_number}
\mathbb E(M_{\Cind})=\frac{1}{\lambda}\left((1+\lambda)\prod_{k\ge0}\left(1+\lambda\,\mathbb E(X^{k+1})\right)-1\right).
\end{equation}
\end{theorem}

In contrast with $C_{\mathrm{cat}}(\lambda,\nu)$, $C_{\mathrm{ind}}(\lambda,\nu)$ achieves the intended effect, as it exhibits a phase transition phenomenon. This behavior is clearly illustrated in the explicit examples below: by tuning an appropriate parameter, \eqref{eq:prob_surv} can be zero or strictly positive.
Note that in the case where $\nu=\delta_p$ for some $p\in[0,1]$, then we have $\mathbb E(X^{k+1})=p^{k+1}$ and thus our expression \eqref{eq:exp_number} for the expected number of catastrophes corresponds to formula~\eqref{Tiemp_Medio_Artalejo} 

Just as for \eqref{eq:exp_cat_random}, expressions \eqref{eq:prob_surv} and \eqref{eq:exp_number} are hardly computable in general. The following corollary provides easily verifiable conditions for survival, extinction and finiteness of the expected number of catastrophes. Explicit examples will be considered later on.

\begin{cor}\label{cor:criteria}
Consider the process $\Cind(\lambda,\nu)$ and let $X\sim \nu$. 
\begin{enumerate}
    \item \textit{General criteria.}
    \begin{itemize}
        \item If $\displaystyle\liminf_{n\to\infty} n\lambda\,\mathbb E(X^{n+1})>1$, then $\mathbb P(\mathcal A_{\Cind})>0$.
        \item If $\displaystyle\limsup_{n\to\infty} n\lambda\,\mathbb E(X^{n+1})<1$, then $\mathbb P(\mathcal A_{\Cind})=0$.
    \end{itemize}
    Furthermore, the expected number of catastrophes is infinite if and only if the sum of the moments diverges; that is,
    \begin{equation*}
        \mathbb E(M_{\Cind})=\infty \quad \text{if and only if} \quad \sum_{n\ge0}\mathbb E(X^{n+1})=\infty.
    \end{equation*}
    \item \textit{Density criteria.} Suppose that $\nu$ admits a bounded density $f$ on $[0,1]$ and that the left limit $f(1^-):=\lim_{x\uparrow 1} f(x)$ exists. Then:
    \begin{itemize}
        \item If $f(1^-)>\frac{1}{\lambda}$, then $\mathbb P(\mathcal A_{\Cind})>0$.
        \item If $f(1^-)<\frac{1}{\lambda}$, then $\mathbb P(\mathcal A_{\Cind})=0$.
    \end{itemize}
    If, in addition, $f$ is left-differentiable at $1$ and $f(1^-)=\frac{1}{\lambda}$, then $\mathbb P(\mathcal A_{\Cind})=0$.    
\end{enumerate}
\end{cor}
\begin{obs}
Note that, while the positivity of $\mathbb{P}(\mathcal{A}_{\Cind})$ depends on both $\lambda$ and the asymptotic behavior of the moments, the finiteness of $\mathbb{E}(M_{\Cind})$ is independent of $\lambda$ and determined solely by the summability of the moment sequence. 
\end{obs}
Intuitively, it is the behavior of the distribution $\nu$ close to $1$ that determines the fate of the population. This phenomenon is natural and was similarly noticed in \cite{CM2025}, although in the context of a different model. Furthermore, our criteria leave room for interesting scenarios where the population is guaranteed to die out, but the expected number of catastrophes until extinction is infinite. We will explicitly observe this behavior in the upcoming examples. 
In these examples,  we use specific  survival parameter distributions yielding simple expressions for the exact survival probability, allowing to identify critical phase transition thresholds, and to analyze the expected number of catastrophes.

\begin{exa}[Beta distribution]
Consider the process $\Cind(\lambda,\nu)$, where $\nu$ is the $\text{Beta}(a, b)$ distribution with parameters $a, b > 0$. Let $X \sim \text{Beta}(a, b)$. The density of $X$ is given by
\[
f(x) = \frac{x^{a-1}(1-x)^{b-1}}{B(a,b)}, \quad \text{for } x \in (0,1),
\]
where $B(a,b)$ is the Beta function. The $(n+1)$-th moment of $X$ is given by
\begin{equation*}
\mathbb{E}[X^{n+1}] = \frac{B(n+a+1, b)}{B(a, b)} = \frac{\Gamma(n+a+1)\Gamma(a+b)}{\Gamma(n+a+b+1)\Gamma(a)}.
\end{equation*}

According to Wendel's Theorem \cite{W1948}, for any $s \in \mathbb{R}$,
\begin{equation*}
\lim_{x \to \infty} \frac{\Gamma(x+s)}{x^s \Gamma(x)} = 1,
\end{equation*}
which implies that as $n \to \infty$,
\begin{equation*}
\frac{\Gamma(n+a+1)}{\Gamma(n+a+b+1)} \sim n^{-b}.
\end{equation*}
Substituting this asymptotic behavior back into the moment expression yields
\begin{equation*}
\lim_{n \to \infty} n \lambda\,\mathbb{E}[X^{n+1}] = \frac{\lambda \Gamma(a+b)}{\Gamma(a)} \lim_{n \to \infty} n^{1-b} = 
\begin{cases} 
\infty & \text{if } 0 < b < 1, \\ 
a\lambda & \text{if } b = 1, \\ 
0 & \text{if } b > 1.  
\end{cases}
\end{equation*}

\vspace{0.2cm}
\noindent \textit{Expected Number of Catastrophes.}
By the first part of Corollary \ref{cor:criteria}, $\mathbb E(M_{\Cind})<\infty$ if and only if  $b>1$. Although we could not find an explicit closed form for it, substituting the $(k+1)$-th moment into \eqref{eq:exp_number} allows us to exactly express the expected number of catastrophes as the convergent infinite product
\begin{equation*}
\mathbb{E}(M_{\Cind}) = \frac{1}{\lambda}\left( (1+\lambda)\prod_{k=0}^{\infty}\left(1 + \lambda \frac{\Gamma(k+a+1)\Gamma(a+b)}{\Gamma(k+a+b+1)\Gamma(a)}\right) - 1 \right).
\end{equation*}

\vspace{0.2cm}
\noindent \textit{Survival Probability and Phase Transition.}
Let us now focus on the survival probability $\mathbb P(\mathcal A_{\Cind})$. By the first part of Corollary \ref{cor:criteria} (which does not require a bounded density), we easily obtain the survival criteria for $a \neq \frac{1}{\lambda}$. 

The critical case $a = \frac{1}{\lambda}$ can be directly resolved using the second part of the corollary only if the density is bounded on $[0,1]$, which requires $a \ge 1$ (i.e., $\lambda \le 1$). For $a < 1$, the corollary alone is inconclusive at the critical value. We summarize the regime as follows.
\begin{enumerate}
    \item[(a)] If $b \in (0,1)$, then $\mathbb P(\mathcal A_{\Cind}) > 0$.
    \item[(b)] If $b > 1$, then $\mathbb P(\mathcal A_{\Cind}) = 0$.
    \item[(c)] If $b = 1$, then:
    \begin{itemize}
        \item[$\circ$] If $a > \frac{1}{\lambda}$, then $\mathbb P(\mathcal A_{\Cind}) > 0$.
        \item[$\circ$] If $a < \frac{1}{\lambda}$, then $\mathbb P(\mathcal A_{\Cind}) = 0$.
        \item[$\circ$] If $a = \frac{1}{\lambda}$ and $\lambda \le 1$, then $\mathbb P(\mathcal A_{\Cind}) = 0$.
    \end{itemize}
\end{enumerate}

To compute the exact probability when $b=1$ and $a > \frac{1}{\lambda}$, notice that the moments simplify to $\mathbb{E}[X^{k+1}] = \frac{a}{a+k+1}$. Using Theorem \ref{Th:Particula_aleatoria}, we evaluate the infinite sum in the denominator:
\begin{equation*}
\sum_{j=1}^{\infty} \prod_{k=0}^{j-1} \frac{a+k+1}{a(1+\lambda)+k+1} = \sum_{j=1}^{\infty} \prod_{k=0}^{j-1} \frac{\alpha+k}{\gamma+k} = {}_2F_1(\alpha, 1; \gamma; 1) - 1,
\end{equation*}
where we used the notation $\alpha = a+1$ and $\gamma = a(1+\lambda)+1$, and ${}_2F_1$ is the Gaussian hypergeometric function. As long as $\gamma > \alpha + 1$ (which corresponds exactly to $a > \frac{1}{\lambda}$, as we are assuming here), ${}_2F_1(\alpha, 1; \gamma; 1)$ simplifies to $\frac{\gamma-1}{\gamma-\alpha-1}$ (see, e.g., \cite[Eq.~15.1.20]{AS1964}). Substituting this value back into the probability formula yields
\begin{equation*}
    \mathbb{P}(\mathcal{A}_{\Cind}) = 1 - \frac{1}{a\lambda}.
\end{equation*}

\vspace{0.2cm}
\noindent \textit{The Critical Threshold.}
Finally, we can fully resolve the critical case $b=1$ and $a = \frac{1}{\lambda}$ for all $\lambda > 0$. Let $h(\lambda) = \mathbb{P}(\mathcal{A}_{\Cind})$ denote the survival probability as a function of the rate $\lambda$. By coupling arguments, one can see that $h(\lambda)$ is non-decreasing in $\lambda$ for all $\lambda > 0$. For a fixed $a > 0$ and $b=1$, we consider the critical rate $\lambda_c = \frac{1}{a}$. By the non-decreasing property of $h$, its value at $\lambda_c$ is bounded above by its right-hand limit,
\begin{equation*}
    0 \le h(\lambda_c) \le \lim_{\lambda \downarrow \lambda_c} h(\lambda).
\end{equation*}
Using our explicit formula for $\lambda > \frac{1}{a}$, we evaluate this limit:
\begin{equation*}
    \lim_{\lambda \downarrow 1/a} \left( 1 - \frac{1}{a\lambda} \right) = 1 - 1 = 0.
\end{equation*}
This immediately forces $h(1/a) = 0$. Therefore, at the critical value, the process goes extinct almost surely ($\mathbb{P}(\mathcal{A}_{\Cind}) = 0$), completely resolving the gap left by the corollary.
\end{exa}

\begin{obs}
The assumption $b=1$ analyzed above corresponds precisely to the \textit{power-function distribution}. For this  case, the model $\Cind(\lambda,\nu)$ has a critical parameter $\lambda_c = \frac{1}{a}$, such that the model survives if and only if $\lambda > \lambda_c$. Moreover, the survival probability is given by
\[
\mathbb{P}(\mathcal{A}_{\Cind}) = 
\begin{cases} 
1 - \frac{1}{a\lambda} & \text{if } \lambda > \lambda_c, \\ 
0 & \text{if } \lambda \le \lambda_c. 
\end{cases}
\]
Setting $a=1$ further reduces this to the \textit{standard uniform distribution}.
\end{obs}

\begin{exa}[Truncated Exponential Distribution]
Consider the process $\Cind(\lambda,\nu)$, where $\nu$ follows a \textit{truncated exponential distribution} on the interval $[0,1]$ with rate parameter $\gamma > 0$. The density of $\nu$ is given by
\[
f(x) = \frac{\gamma e^{-\gamma x}}{1 - e^{-\gamma}}, \quad \text{for } x \in [0,1].
\]

\vspace{0.2cm}
\noindent \textit{Survival Probability and Phase Transition.}
Since $f$ is continuous and strictly bounded on $[0,1]$, and differentiable at $x=1$, we can directly apply the second part of Corollary \ref{cor:criteria}. We first evaluate the left limit of the density at $1$:
\begin{equation*}
f(1^-) = \lim_{x\uparrow 1} \frac{\gamma e^{-\gamma x}}{1 - e^{-\gamma}} = \frac{\gamma}{e^\gamma - 1}.
\end{equation*}
By Corollary \ref{cor:criteria}, the critical threshold is determined by the condition $f(1^-) = \frac{1}{\lambda}$, which yields the critical rate $\lambda_c = \frac{e^\gamma - 1}{\gamma}$. Because $f$ is differentiable at $1$, the corollary guarantees that the process goes extinct almost surely at this critical value ($\lambda = \lambda_c$). Thus, we can summarize the survival criterion with a single equivalence:
\begin{equation*}
\mathbb{P}(\mathcal{A}_{\Cind}) > 0 \quad\text{if and only if}\quad \lambda > \frac{e^\gamma - 1}{\gamma}.
\end{equation*}
This illustrates that for a rapidly decaying density (i.e., when $\gamma$ is large), the critical rate $\lambda_c$ grows exponentially. Consequently, a significantly higher rate $\lambda$ is required for the process to survive.

\vspace{0.2cm}
\noindent \textit{Expected Number of Catastrophes.}
To determine the expected number of catastrophes, we evaluate the sum of the moments $\sum_{n \ge 0} \mathbb{E}[X^{n+1}]$. Notice that the density $f(x)$ of the truncated exponential distribution is strictly decreasing on $[0,1]$. Therefore, it is strictly bounded from below by its minimum value at $x=1$, yielding
\begin{equation*}
f(x) \ge f(1) = \frac{\gamma e^{-\gamma}}{1-e^{-\gamma}} > 0.
\end{equation*}
This allows us to lower bound the moments directly to obtain
\begin{equation*}
\mathbb{E}[X^{n+1}] = \int_0^1 x^{n+1} f(x)\,dx \ge f(1) \int_0^1 x^{n+1}\,dx = \frac{f(1)}{n+2}.
\end{equation*}
Since $f(1)$ is a strictly positive constant, the sum of the moments is bounded below by a harmonic series, which diverges. By the first part of Corollary \ref{cor:criteria}, the divergence of $\sum_{n \ge 0} \mathbb{E}[X^{n+1}]$ unconditionally implies that for any $\gamma > 0$ we have
\begin{equation*}
\mathbb{E}(M_{\Cind}) = \infty.
\end{equation*}
This means that whenever $\lambda \le \lambda_c$, the population is guaranteed to die out ($\mathbb{P}(\mathcal{A}_{\Cind}) = 0$), yet the expected number of catastrophes until extinction is infinite.
\end{exa}

\section{Bridging with the Firework process}\label{sec:bridge}

We believe that an important contribution of the present work is to connect two different bodies of the literature by bridging the so-called Firework Processes (FP) with what we call General Immigration Processes with Catastrophes (GIPC). 

Let us mention in passing that this parallel could be systematically used to import results from each literature that could be novel in the other. For the sake of conciseness, we will only this bridge, in Section \ref{sec:proofs}, to derive the survival probability formula  and the expected time to extinction formula obtained in Theorem~\ref{Th:Particula_aleatoria}. 

This section starts by defining both processes, FP and GIPC, and then shows how they are connected.

\subsection{The Firework process}

Consider a rumor spreading process on $\mathbb{N}=\{0,1,2,\ldots\}$.
At each site $n\in\mathbb{N}$ there is exactly one individual (later on, we will explain how the case with a random number $N_n,n\in\mathbb{N}$ of individuals derives from this simplest case). 
Initially, all individuals are ignorant, except for the one located at site $0$. Let $R\in\mathbb N$ be a random variable interpreted as a radius of information, and associate to each individual $n\ge0$ an independent copy $R_n\sim R$.

The Firework process is defined recursively as follows.
Let $A_n$ denote the set of sites whose individuals have been informed
by stage $n$, so that $A_0=\{0\}$. For $n\ge 1$, we set
\[
A_n=\Big\{ i\in\mathbb{N}:\ \exists\, j\in A_{n-1}\ \text{such that }
i\in\{j,\ldots,j+R_j\}\Big\}\setminus A_{n-1}.
\]

Once informed, individuals become spreaders and remain spreaders forever.
The set $\bigcup_{n\ge 0}A_n$ corresponds to the set of spreaders at the end of the spreading procedure. Define the
final range of the rumor as
\[
M_F:=\Big|\bigcup_{n\ge 0}A_n\Big|.
\]

The Firework process is said to \emph{survive} if the event $\mathcal A_F := \{M_F=\infty\}$ occurs with positive probability. The following result gives an exact expression for this survival probability and for the expected final range of the rumor $\mathbb E(M_F)$.
\begin{theorem}\label{Th:Gallo}
Consider the Firework process on $\mathbb{N}$ as defined above, with $\alpha_k:=\mathbb P(R\le k),k\ge0$ being the radius distribution. Then
\begin{enumerate}
\item $\mathbb P(\mathcal A_F)=\left(1+\sum_{j\ge1}\prod_{k=0}^{j-1}\alpha_k\right)^{-1}$ and,
\item $\mathbb E (M_F)=\left(\prod_{k\ge0}\alpha_k\right)^{-1}$ (with $\frac{1}{0}:=+\infty$).
\end{enumerate}
\end{theorem}
Statement (1) was proved by Gallo et al.~\cite{GGJR2014}. As far as we know, Statement (2) is new in the literature of FP. We will give a proof using renewal theory, following similar arguments as in \cite{GGJR2014}. 
\begin{obs}\label{FP_general}
We considered above the one individual per site case as was originally defined in \cite{JMZ2011} and \cite{GGJR2014}. But Theorem \ref{Th:Gallo}  can be used when the process starts with a random number of individual per site. In this case, each individual $i$ at site $n$ is endowed with an independent spreading radius $\bar R_n^{\,i}$, where the family $\{\bar R_n^{\,i} : n\ge 0,\ i\ge 1\}$ consists of i.i.d.\ $\mathbb{N}$-valued random variables with distribution $\bar\alpha_k:=\mathbb{P}(\bar R_0^{\,1}\le k),  k\ge 0$. We define the \emph{effective spreading radius} of site $n$ as $R_n:=\sup_{1\le i\le N_n}\bar R_n^{\,i}$, with the convention that $R_n=0$ when $N_n=0$. Now everything in the model with random number of individuals works as if there were exactly one individual per site with radius $R_n$ and Theorem \ref{Th:Gallo} applies. We will keep the notation $\mathcal A_F$ and $M_F$  even in the case of having a random number of individual at each site.    
\end{obs}

 \subsection{General Immigration Processes with Catastrophes} Suppose catastrophes and immigrations occur as before, but the way catastrophes affect the individual is different:  we associate to each individual, upon birth, an independent number of catastrophes it will survive, having common distribution $\mathfrak r$ on $\mathbb N$. We denote this general model by GIPC($\lambda, \mathfrak r$). Notice in particular that, in this model, a given catastrophe may be unable to affect a given individual if it is not is ``moment to die''. If $\mathfrak r$ is geometric then the GIPC($\lambda,\mathfrak r$) is an IPBC. Depending on $\mathfrak r$, the process $\mathrm{GIPC}(\lambda,\mathfrak r)$ either survives (in the sense of \eqref{def_sobrevivencia}) or becomes extinct.

\subsection{The GIPC seen as a FP}\label{GIPC-FP}
The GIPC admits an alternative description as a spatial process on $\mathbb{N}$ as follows.  
Start the process with a single colony located at vertex $0$ (the root of $\mathbb{N}$), initially containing one individual. The colony grows according to a Poisson process with rate $\lambda$ until the first catastrophe occurs (after an exponential time of rate $1$). 
Right after this moment, the surviving individuals (those haven't still reached their maximal number of catastrophes) simultaneously move to the neighboring site to the right, where they form a new colony and continue the evolution of the process in the same way. 

With this spatial interpretation, we can now parallel the GIPC($\lambda,\mathfrak r$) with the Firework process having $\bar R_n^i\sim \mathfrak r, N_n\sim \mathrm{Geo}\!\left(\frac{1}{1+\lambda}\right)$ (all independent).

Observe that the catastrophe process starts with $N_0 \sim 1+N$, rather than 
$N_0 \sim N$. 
However, by the memoryless property of 
$N_0 \sim \mathrm{Geo}\!\left(\frac{1}{1+\lambda}\right)$, we obtain
\[
\mathbb{P}(\mathcal A_C)
= \mathbb{P}(\mathcal A_F \mid N_0 \ge 1).
\]
Since
\[
\mathbb{P}(\mathcal A_F)
= \mathbb{P}(\mathcal A_F \mid N_0 \ge 1)\,\mathbb{P}(N_0 \ge 1),
\]
it follows that
\begin{equation}\label{Relacion_FP}
\mathbb{P}(\mathcal A_C)
= \frac{\mathbb{P}(\mathcal A_F)}{\mathbb{P}(N_0 \ge 1)}
= \frac{(1+\lambda)\mathbb{P}(\mathcal A_F)}{\lambda}.
\end{equation}
Similarly
\[\mathbb E(M_C)=\mathbb{E}(M_F \mid N_0 \ge 1)\]
so 
\begin{align*}
\mathbb E(M_F)&=\mathbb{E}(M_F \mid N_0 \ge 1)\mathbb P(N_0\ge1)+\mathbb{E}(M_F \mid N_0=0)\mathbb P(N_0=0)\\
&=\mathbb E(M_C)\,\,\mathbb P(N_0\ge1)+\mathbb P(N_0=0)
\end{align*}
and thus
\begin{equation}\label{Relacion_FPExpect}
\mathbb E(M_C)=\frac{1}{\lambda}\left((1+\lambda)\mathbb E(M_F)-1\right).
\end{equation}

\section{Proofs}\label{sec:proofs}

\begin{proof}[Proof of Theorem \ref{Th:Catastrofe_aleatoria}]
\noindent \textit{Part (1).}
The proof relies on Foster's theorem, which we recall below. For a proof and further discussion, see Fayolle \textit{et al.}~\cite[Theorem~2.2.3]{FMM1995}.

\begin{theorem}[Foster's theorem]
Let $\{W_n\}_{n\geq 0}$ be an irreducible and aperiodic Markov chain on a countable state space $\mathcal{S}=\{\alpha_i,\ i\geq0\}$. Then $\{W_n\}_{n\geq 0}$ is ergodic if and only if there exist a positive function $g:\mathcal{S}\to(0,\infty)$, a constant $\varepsilon>0$, and a finite set $A\subset\mathcal{S}$ such that
\[
\mathbb{E}\big[g(W_{n+1})-g(W_n)\mid W_n=\alpha_j\big]\le -\varepsilon,
\qquad \alpha_j\notin A,
\]
and
\[
\mathbb{E}\big[g(W_{n+1})\mid W_n=\alpha_i\big]<\infty,
\qquad \alpha_i\in A.
\]
\end{theorem}

Let $\{Y_n\}_{n\geq 0}$ be the discrete-time Markov chain embedded in the process $C(\lambda,\nu)$, with transition probabilities
\[
P_{i,i+1}=\frac{\lambda}{\lambda+1},
\qquad
P_{i,j}=\frac{1}{\lambda+1}\binom{i}{j}
\int_0^1 p^j(1-p)^{i-j}\nu(dp),
\quad 0\le j\le i,
\]

Ergodicity of $\{Y_n\}$ implies that the extinction time of $C(\lambda,\nu)$ has finite mean. Observe that $\{Y_n\}$ is irreducible and aperiodic. We apply Foster's theorem to show that $\{Y_n\}$ is ergodic for any probability distribution $\nu$ on $[0,1]$.

Consider the function $g:\mathbb{N}\to(0,\infty)$ defined by $g(i)=i+1$, and fix $\varepsilon>0$. Define the set
\[
A:=\left\{ i\in\mathbb{N} :
\frac{1}{\lambda+1}\big(\lambda+i(\mu_p-1)\big)>-\varepsilon \right\},
\]
where $\mu_p=\int_0^1 p\,\nu(dp)$.
Since $\mu_p<1$, the set $A$ is finite.

Since $g(j)-g(i)=j-i$, we compute
\[
\mathbb{E}[g(Y_{n+1})-g(Y_n)\mid Y_n=i]
=
\sum_{j=0}^{i+1}(j-i)P_{i,j} 
=
\frac{1}{\lambda+1}\big(\lambda+i(\mu_p-1)\big).
\]

Thus,
\[
\mathbb{E}[g(Y_{n+1})-g(Y_n)\mid Y_n=i]
\le -\varepsilon,
\qquad i\notin A.
\]

Moreover,
\[
\mathbb{E}[g(Y_{n+1})\mid Y_n=i]
=
\sum_{j=0}^{i+1} g(j)P_{i,j}
=
\frac{1}{\lambda+1}
\big(\lambda(i+2)+i\mu_p+1\big)
<\infty,
\qquad i\in A.
\]

Therefore, all the conditions of Foster's theorem are satisfied, and $\{Y_n\}$ is ergodic. This concludes the proof.

\noindent
\textit{Part (2).} Let $\tau_i = \mathbb{E}[M_{C_{\mathrm{cat}}} \mid X(0)=i]$ be the expected number of catastrophes until extinction starting from $i$ individuals. Clearly, $\tau_0 = 0$. By conditioning on the first event (either an immigration or a catastrophe), the sequence $\{\tau_i\}_{i\ge 1}$ satisfies the backward Kolmogorov equations:
\[
(\lambda + 1)\tau_i = 1 + \lambda \tau_{i+1} + \int_0^1 \sum_{j=0}^i \binom{i}{j} p^j (1-p)^{i-j} \tau_j \,\nu(dp), \quad i \ge 1.
\]
Rearranging terms, we obtain
\begin{equation}\label{eq:kolmogorov_tau}
\lambda (\tau_{i+1} - \tau_i) - \tau_i + \int_0^1 \sum_{j=0}^i \binom{i}{j} p^j (1-p)^{i-j} \tau_j \,\nu(dp) = -1.
\end{equation}
To solve this infinite system, we introduce the Poisson transform of the sequence $\{\tau_i\}$, defined by
\[
F(z) = \sum_{i=0}^\infty \tau_i e^{-z} \frac{z^i}{i!}.
\]
\begin{obs}\label{Rem:ConvergenceF}
The convergence of the Poisson transform $F(z)$ for any $z > 0$ is guaranteed by the structure of the backward Kolmogorov equations. Since the expected times are non-negative ($\tau_j \ge 0$), we can drop the integral and the negative constant from \eqref{eq:kolmogorov_tau} to obtain the strict upper bound
\[
\lambda \tau_{i+1} < (\lambda + 1)\tau_i, \quad i \ge 1.
\]
Iterating this geometric recurrence yields
\[
\tau_i < \tau_1 \left( \frac{\lambda + 1}{\lambda} \right)^{i-1} = K C^i,
\]
where $C = \frac{\lambda + 1}{\lambda}$ and $K = \frac{\tau_1}{C}$. Substituting this exponential bound into the definition of $F(z)$ yields
\[
F(z) = \sum_{i=1}^\infty \tau_i e^{-z} \frac{z^i}{i!} < K e^{-z} \sum_{i=1}^\infty \frac{(Cz)^i}{i!} = K e^{-z} (e^{Cz} - 1).
\]
Since $C$ and $K$ are finite constants, $F(z)$ is strictly bounded by $K e^{z(C-1)}$, proving that the series converges absolutely. 
\end{obs}
Note that $F(0) = \tau_0 = 0$ and $F'(0) = \tau_1$, which is the quantity we wish to compute since the process starts with a single individual ($X(0)=1$). Multiplying \eqref{eq:kolmogorov_tau} by $e^{-z} \frac{z^i}{i!}$ and summing over $i \ge 1$, the binomial thinning property yields the differential-functional equation
\[
\lambda F'(z) - F(z) + \int_0^1 F(pz) \,\nu(dp) = e^{-z} - 1 + \lambda \tau_1 e^{-z}.
\]
Since $F(0)=0$, we can express $F(z)$ as a power series $F(z) = \sum_{k=1}^\infty c_k z^k$. Substituting this series into the equation and matching the coefficients of $z^k$ gives
\[
\lambda (k+1) c_{k+1} - c_k + \mathbb{E}[X^k] c_k = (1 + \lambda \tau_1) \frac{(-1)^k}{k!}, \quad k \ge 1.
\]
This is a first-order linear recurrence relation for the coefficients $c_k$. Let $v_k = (-1)^{k-1} k! c_k$, which implies $v_1 = c_1 = \tau_1$. The recurrence simplifies to
\begin{equation}\label{eq:recurrence_v}
v_{k+1} = - \frac{1 - \mathbb{E}[X^k]}{\lambda} v_k + \frac{1 + \lambda \tau_1}{\lambda}.
\end{equation}
Iterating \eqref{eq:recurrence_v} yields the general solution
\begin{eqnarray}\label{eq:factored_v}
v_{k+1} &=& (-1)^k v_1 \prod_{j=1}^k \frac{1-\mathbb{E}[X^j]}{\lambda} + \frac{1+\lambda \tau_1}{\lambda} \sum_{n=1}^k (-1)^{k-n} \prod_{j=n+1}^k \frac{1-\mathbb{E}[X^j]}{\lambda} \nonumber\\
&=& \left( (-1)^k \prod_{j=1}^k \frac{1-\mathbb{E}[X^j]}{\lambda} \right) \left[ v_1 + \frac{1+\lambda \tau_1}{\lambda} \sum_{n=1}^k \frac{(-\lambda)^n}{\prod_{j=1}^n (1-\mathbb{E}[X^j])} \right].
\end{eqnarray}
To mathematically select the unique physical solution from this recurrence, we recall that the Poisson transform $F(z) = \sum_{k=1}^\infty c_k z^k$ relates to the original sequence via $e^z F(z) = \sum_{i=1}^\infty \tau_i \frac{z^i}{i!}$. Expanding $e^z$ as a power series and matching the coefficients of $z^i$ yields the exact algebraic recovery formula
$$\tau_i = \sum_{k=1}^i \binom{i}{k} k! c_k = \sum_{k=1}^i \binom{i}{k} (-1)^{k-1} v_k.$$
Substituting \eqref{eq:factored_v} into this identity, we obtain
$$\tau_i = \sum_{k=1}^i \binom{i}{k} \left( \prod_{j=1}^{k-1} \frac{1-\mathbb{E}[X^j]}{\lambda} \right) S_{k-1},$$
where $S_{k-1}$ denotes the bracketed term in \eqref{eq:factored_v} evaluated at step $k-1$. Assuming for the moment that $\lambda < 1$ (which guarantees that the infinite series within the bracket converges), the asymptotic growth and sign of $\tau_i$ are entirely dictated by the limit of the bracket, $S_\infty = \lim_{k\to\infty} S_k$. 
If $S_\infty < 0$, the negative terms will eventually dominate the sum, forcing $\tau_i < 0$ for large $i$, which is impossible for an expected time. Hence, any probabilistically valid solution strictly requires $S_\infty \ge 0$. By the general theory of Markov chains (see, e.g., Norris \cite[Theorem~1.3.5]{Norris1997}), the expected hitting time to the absorbing state is uniquely identified as the \emph{minimal non-negative solution} to the associated system of linear equations. Since $S_k$ is strictly increasing with respect to the initial condition $\tau_1$, obtaining the minimal solution dictates setting the limit to its lowest valid bound, $S_\infty = 0$. This rigorous minimality argument enforces the exact boundary condition:
\[
v_1 + \frac{1+\lambda \tau_1}{\lambda} \sum_{n=1}^\infty \frac{(-\lambda)^n}{\prod_{j=1}^n (1-\mathbb{E}[X^j])} = 0.
\]
Recalling that $v_1 = \tau_1$ and substituting the definition of $\mathcal{S}_\nu(\lambda)$, we obtain
\[
\tau_1 + \frac{1+\lambda \tau_1}{\lambda} \mathcal{S}_\nu(\lambda) = 0.
\]
Solving for $\tau_1$ yields the explicit formula for $\lambda < 1$:
\[
\mathbb E(M_{C_{\mathrm{cat}}}) = \tau_1 = -\frac{1}{\lambda} \frac{\mathcal{S}_\nu(\lambda)}{1 + \mathcal{S}_\nu(\lambda)}.
\]
This restriction on $\lambda$ is then lifted via analytic continuation, extending the validity of the formula to all $\lambda > 0$ (see Remark \ref{Rem:AnalyticContinuation}).
\end{proof}

\begin{obs}\label{Rem:AnalyticContinuation}
Regarding the analytical properties of $\mathcal{S}_\nu(\lambda)$, two technical points are in order. First, since $\nu$ is not concentrated at $1$, the dominated convergence theorem implies $\lim_{n \to \infty} \mathbb{E}[X^{n+1}] = 0$. By the ratio test, the series \eqref{Eq:aux} has a radius of convergence equal to $1$, meaning it diverges for $\lambda \ge 1$. 
Second, to justify the extension to $\lambda \ge 1$, we note that $\mathbb E(M_{C_{\mathrm{cat}}})$ is an analytic function of $\lambda$ on $(0,\infty)$. Recall the backward Kolmogorov system \eqref{eq:kolmogorov_tau} for the vector of expected times $\tau = (\tau_1, \tau_2, \dots)^T$. This infinite system can be recast in operator form as $(A + \lambda B)\tau = -\mathbf{1}$, where $\mathbf{1}$ is the all-ones vector, and the linear operators $A$ and $B$ are strictly independent of $\lambda$. Specifically, their components are given by:
\[
(A\tau)_i = \int_0^1 \sum_{j=1}^i \binom{i}{j} p^j (1-p)^{i-j} \tau_j \,\nu(dp) - \tau_i, \quad \text{and} \quad (B\tau)_i = \tau_{i+1} - \tau_i.
\]
By classical perturbation theory \cite{Kato1995}, the expected time vector $\tau = -(A+\lambda B)^{-1}\mathbf{1}$ is an analytic function of $\lambda$ wherever the inverse is a bounded operator. Since the Foster criterion (Theorem \ref{Th:Catastrofe_aleatoria}, Part 1) guarantees a finite expected time for all $\lambda > 0$, this inverse is well-defined on $(0,\infty)$. Thus, the algebraic relation established for $\lambda \in (0,1)$ extends uniquely to all $\lambda > 0$ via analytic continuation.  The practical realization of this extension is explicitly illustrated in Example \ref{Ex:PowerFunction}, where the derived closed-form expressions remain perfectly well-behaved for any $\lambda > 0$.
\end{obs}

\begin{proof}[Proof of Theorem \ref{Th:Particula_aleatoria}]
For the proof, we make use of the bridge with the Firework process (Section~\ref{GIPC-FP}).
Let $N\sim \mathrm{Geo}(1/(1+\lambda))$ and let $(N_i)_{i\ge0}$ be an i.i.d.\ collection of random variables with $N_i\sim N$, denoting the number of particles at site $i$ when the firework process starts. Assume that each particle has spreading radius distribution $\bar{\alpha}_i=\mathbb{P}(\bar{R}\le i)$, $i\ge0$. Using Theorem~\ref{Th:Gallo} and Remark~\ref{FP_general}, a simple calculation shows that for this Firework process 
\[
\mathbb P(\mathcal A_F)
=
\left(1+\sum_{j\ge1}\prod_{k=0}^{j-1}\frac{1}{1+\lambda(1-\bar{\alpha}_k)}\right)^{-1}\,\,\text{and}\,\,\,\mathbb E (M_F)=\prod_{k\ge0}\frac{1}{1+\lambda(1-\bar{\alpha}_k)}
\]
In the model $\Cind(\lambda,\nu)$, the number of catastrophes survived by a particle $u$ born at site $i\geq 0$ follows a geometric distribution with parameter $X_{i,u}\sim \nu$, where the family $(X_{i,u})_{i\ge 0,\,u\ge 1}$ is independent. It follows that
\[
1-\bar{\alpha}_k
=
\int_0^1 x^{k+1}\nu(dx)
=
\mathbb E(X^{k+1}).
\]
The expressions of $\mathbb P(\mathcal A_{\Cind})$ and $\mathbb E(M_{\Cind})$ now follow from relations~\eqref{Relacion_FP} and \eqref{Relacion_FPExpect}.
\end{proof}

\begin{proof}[Proof of Corollary~\ref{cor:criteria}]
\noindent \textit{Part (1).}
We begin with the statement concerning the expected number of catastrophes. From Theorem \ref{Th:Particula_aleatoria}, the expectation is given by
\[
\mathbb E(M_{\Cind}) = \frac{1}{\lambda}\left((1+\lambda)\prod_{k\ge0}\left(1+\lambda\,\mathbb E(X^{k+1})\right)-1\right).
\]
Since $1+\lambda\,\mathbb E(X^{k+1}) \ge 1$ for all $k$, we can use the standard property of infinite products (see, e.g., \cite[Chapter VII]{Knopp1990}) which states that for any sequence $c_k \ge 0$, $\prod_{k\ge0} (1+c_k) = \infty$ if and only if $\sum_{k\ge0} c_k = \infty$.
Setting $c_k = \lambda\,\mathbb E(X^{k+1})$ immediately yields that $\mathbb E(M_{\Cind}) = \infty$ if and only if $\sum_{n\ge0} \mathbb E(X^{n+1}) = \infty$.

Next, we address the survival probability. By Theorem \ref{Th:Particula_aleatoria}, survival occurs with positive probability if and only if the infinite series $\sum_{j\ge1}\prod_{k=0}^{j-1}\frac{1}{1+\lambda\mathbb E(X^{k+1})}$  converges. Let
\[
a_j := \prod_{k=0}^{j-1} \frac{1}{1+\lambda\,\mathbb E(X^{k+1})}, \quad j\ge 1.
\]
Survival is equivalent to the convergence of $\sum_{j\ge1} a_j$. Observe that
\[
j\left(\frac{a_j}{a_{j+1}}-1\right) = j\lambda\,\mathbb E(X^{j+1}).
\]
According to Raabe's test (see, e.g., \cite{Knopp1990}), if the limit inferior of this sequence is strictly greater than $1$, the series converges (hence $\mathbb P(\mathcal A_{\Cind})>0$). Conversely, if the limit superior is strictly less than $1$, the series diverges (hence $\mathbb P(\mathcal A_{\Cind})=0$). This concludes the proof of the first part.

\vspace{0.2cm}
\noindent \textit{Part (2).}
Suppose now that $\nu$ has a bounded density $f$ on $[0,1]$ and that the left limit $f(1^-):=\lim_{x \uparrow 1} f(x)$ exists. Then
\[
n\mathbb{E}[X^{n+1}] = n\int_0^1 x^{n+1}f(x)\,dx.
\]
Performing the change of variables $x=1-\frac{t}{n}$, we obtain
\[
n\mathbb{E}[X^{n+1}] = \int_0^\infty \mathbf 1_{\{t\le n\}} \left(1-\frac{t}{n}\right)^{n+1} f\!\left(1-\frac{t}{n}\right)\,dt.
\]
For each fixed $t\ge0$, we have $1-\frac{t}{n} \uparrow 1$ as $n\to\infty$. Hence
\[
\mathbf 1_{\{t\le n\}} \to 1, \qquad \left(1-\frac{t}{n}\right)^{n+1} \to e^{-t}, \qquad f\!\left(1-\frac{t}{n}\right) \to f(1^-),
\]
and therefore the integrand converges pointwise to $f(1^-)e^{-t}$.

Since $f$ is bounded, there exists $K>0$ such that $|f(x)|\le K$ for all $x\in[0,1]$. Moreover,
\[
0 \le \mathbf 1_{\{t\le n\}} \left(1-\frac{t}{n}\right)^{n+1} \le e^{-t} \quad \text{for all } t\ge0.
\]
Hence, the integrand is bounded in absolute value by $K e^{-t}$, which is integrable on $[0,\infty)$. The dominated convergence theorem therefore yields
\[
\lim_{n\to\infty} n\mathbb{E}[X^{n+1}] = \int_0^\infty f(1^-)e^{-t}\,dt = f(1^-),
\]
which, combined with part (1), proves the density criteria.

For the final statement, because $f$ is bounded and left-differentiable at $1$, there exists a constant $C>0$ such that $|f(x)-f(1^-)|\le C(1-x)$ for any $x\in[0,1]$. It follows that
\begin{align*}
\left|\mathbb{E}[X^{n}]-\frac{f(1^-)}{n+1}\right| &= \left|\int_{0}^{1} x^n f(x) \, dx - \int_0^1 x^n f(1^-)\,dx \right| \\
&\le \int_{0}^{1} x^n |f(x)-f(1^-)| \,dx \\
&\le C\int_{0}^{1} x^n (1-x) \,dx = \frac{C}{(n+1)(n+2)}.
\end{align*}
We now have
\[
\frac{a_j}{a_{j+1}} = 1+\lambda \mathbb E(X^{j+1}) = 1+\frac{\lambda f(1^-)}{j} + \mathcal{O}\left(\frac{1}{j^2}\right).
\]
In the critical case $f(1^-)=\frac{1}{\lambda}$, this becomes $\frac{a_j}{a_{j+1}} = 1+\frac{1}{j} + \mathcal{O}\left(\frac{1}{j^2}\right)$. Using Gauss's test (see, e.g., \cite{Knopp1990}), we conclude that the series $\sum_{j\ge1} a_j$ diverges, meaning $\mathbb P(\mathcal A_{\Cind})=0$.
\end{proof}

We conclude with the proof of Theorem \ref{Th:Gallo} item (2) since item (1) was already proved in \cite{GGJR2014}. 

\begin{proof}[Theorem \ref{Th:Gallo} item (2)]
Lemma 1 of \cite{GGJR2014} states that $\mathbb P(M_F\ge n)=u_n$, where $u_n,n\ge1,$ is the probability that a discrete undelayed renewal sequence renews at time $n$ (see Section 3 in \cite{GGJR2014}). This renewal sequence has inter-arrival distribution given by $f_k=(1-\alpha_{k-1})\prod_{i=0}^{k-1}\alpha_i,k\ge1$.

On the other hand, observe that
\[
\mathbb E(M_F)=\sum_{n\ge1}\mathbb P(M_F\ge n)=\sum_{n\ge1}u_n.
\]

It is well-known from renewal theory that transience, which corresponds to $f_\infty:=1-\sum_{k\ge1}f_k>0$ is equivalent to $\sum_{n\ge1}u_n<\infty$. In the present case, a simple calculation shows that $f_\infty=\prod_{k\ge0}\alpha_k$. So we can already conclude one part of the statement of Theorem \ref{Th:Gallo} item (2), which is that 
\[
\mathbb E(M_F)=\infty \quad\text{if and only if}\quad\prod_{k\ge0}\alpha_k=0.
\]
It remains to calculate $\mathbb E(M_F)$ under recurrence. This follows also directly from renewal theory (see for instance \cite[Theorem 2 in XIII.3]{Feller1968})
\[
\sum_{n\ge1}u_n=\frac{1}{1-\sum_{n\ge1}f_n}=\frac{1}{\prod_{k\ge0}\alpha_k}.
\]
\end{proof}

\end{document}